\documentclass[11pt]{amsart}
\usepackage{amssymb, amstext, amscd, amsmath}
\usepackage{mathtools, xypic, color}
\usepackage{enumitem}
\usepackage{xcolor}
\usepackage{hyperref}
\theoremstyle{plain}
\newtheorem{thm}{Theorem}[section]

\newtheorem{cor}[thm]{Corollary}

\newtheorem{theorem}[thm]{Theorem}
\newtheorem{proposition}[thm]{Proposition}
\newtheorem{lemma}[thm]{Lemma}
\newtheorem{corollary}[thm]{Corollary}

\newtheorem{question}[thm]{Question}
\theoremstyle{definition}

\mathtoolsset{centercolon}
\newcommand{\bB}{{\mathbb{B}}}
\newcommand{\bC}{{\mathbb{C}}}
\newcommand{\bD}{{\mathbb{D}}}

\newcommand{\bS}{{\mathbb{S}}}
\newcommand{\bT}{{\mathbb{T}}}

  \newcommand{\A}{{\mathcal{A}}}
  \newcommand{\B}{{\mathcal{B}}}
  \newcommand{\C}{{\mathcal{C}}}

  \newcommand{\I}{{\mathcal{I}}}
  \newcommand{\J}{{\mathcal{J}}}

  \newcommand{\M}{{\mathcal{M}}}

\renewcommand{\S}{{\mathcal{S}}}

\newcommand{\rC}{\mathrm{C}}

\newcommand{\ep}{\varepsilon}
\newcommand{\eps}{\varepsilon}
\renewcommand{\phi}{\varphi}

\newcommand{\upchi}{{\raise.35ex\hbox{$\chi$}}}

\newcommand{\Ba}{{\mathbf{a}}}

\newcommand{\lip}{\langle}
\newcommand{\rip}{\rangle}
\newcommand{\ip}[1]{\lip #1 \rip}

\newcommand{\ol}{\overline}

\newcommand{\AB}{{\mathrm{A}(\mathbb{B}_d)}}
\newcommand{\AD}{{\mathrm{A}(\mathbb{D})}}

\newcommand{\FOR}{\text{ for }}

\newcommand{\qand}{\quad\text{and}\quad}

\newcommand{\qfor}{\quad\text{for}\quad}
\newcommand{\qforal}{\quad\text{for all}\quad}

\newcommand{\TS}{\operatorname{TS}}

\begin{document}
\title[Ideals in a multiplier algebra]{Ideals in a multiplier algebra on the ball}

\author[R. Clou\^atre]{Rapha\"el Clou\^atre}
\address{Department of Mathematics, University of Manitoba, 186 Dysart Road,
Winnipeg, MB, Canada R3T 2N2}
\email{raphael.clouatre@umanitoba.ca\vspace{-2ex}}
\thanks{The first author was partially supported by an FQRNT postdoctoral fellowship and a start-up grant from the University of Manitoba.}

\author[K.R. Davidson]{Kenneth R. Davidson}
\address{Department of Pure Mathematics, University of Waterloo, 200 University Avenue West, 
Waterloo, ON, Canada N2L 3G1}
\email{krdavids@uwaterloo.ca}
\thanks{The second author is partially supported by an NSERC grant.}

\begin{abstract}
We study the ideals of the closure of the polynomial multipliers on the Drury-Arveson space. 
Structural results are obtained by investigating the relation between an ideal and its weak-$*$ closure, much in the spirit of the corresponding  classical facts for the disc algebra. Zero sets for multipliers are also considered and are deeply intertwined with the structure of ideals. Our approach is primarily based on duality arguments.
\end{abstract}

\maketitle

\vspace{-3ex} 

\section{Introduction} \label{S:intro}å

We study the ideals and zero sets for the algebra $\A_d$ of multipliers on the Drury-Arveson space
which are limits in the multiplier norm of (analytic) polynomials in $d$ variables.
Precise definitions are given in Section \ref{S:prelim}. 
To a great extent the theory parallels the theory for the ball algebra $\AB$, the uniform closure
of the polynomials in $\rC(\ol{\bB_d})$. However, as the multiplier norm is not comparable to the supremum
norm over the ball, there are complications. Our results are based on duality arguments, and rely heavily on recent machinery developed in \cite{CD} to overcome these complications. It should be noted however that our results on zero sets have (somewhat easier) analogues in the uniform algebra case,
and some of these results appear to be new in that context as well. We emphasize this where appropriate.

The initial motivation for our investigation stemmed from multivariate operator theoretic considerations surrounding the functional calculus associated to an absolutely continuous row contraction. An upcoming paper (\cite{abscont}) deals with this topic.
We focus here on purely function theoretic aspects.

We first review the classical theory.
Let $\bB_d\subset \bC^d$ denote the open unit ball and let $\bS_d$ denote its boundary, the unit sphere. 
The ball algebra $A(\bB_d)$ consists of those functions $\phi$ that are holomorphic on $\bB_d$ and continuous on $\ol{\bB_d}$. It is a uniform algebra when equipped with the supremum norm
\[
\|\phi\|_{\infty}=\sup_{z\in \ol{\bB_d}}|\phi(z)|.
\]
By the maximum modulus principle, it may also be considered as a closed subalgebra of $\rC(\bS_d)$ obtained as the closure of the analytic polynomials. 
Furthermore, let $H^\infty(\bB_d)$ denote the algebra of bounded holomorphic functions on $\bB_d$. 
Functions in $H^\infty(\bB_d)$ have radial limits which are defined on $\bS_d$ almost everywhere with respect to the unique rotation invariant Borel probability measure $\sigma$. In particular, we have the isometric inclusion
\[
H^\infty(\bB_d)\subset L^\infty(\bS_d,\sigma).
\]
In fact, $H^\infty(\bB_d)$ is a weak-$*$ closed subalgebra of $L^\infty(\bS_d,\sigma)$, and
thus it is a dual space that carries the induced weak-$*$ topology.

A natural problem is to determine the ideal structure of $\AB$. First consider the single variable case on the unit disc $\bD\subset \bC$.
A function $\omega\in H^\infty(\bD)$ is called inner if $|\omega|=1$ almost everywhere on the circle $\bS_1$ (with respect to Lebesgue measure). 
Let $X_{\omega}\subset \bS_1$ denote the closure of the points across which $\omega$ cannot be continued holomorphically. This is called the support of $\omega$  (see \cite{Hoffman}). The structure of ideals of $A(\bD)$ was unraveled by Carleson and Rudin independently \cite{Carleson57,Rudin57} (see also \cite{Hoffman} for an exposition of these results).

\begin{theorem}[Carleson, Rudin]\label{T:carlesonrudin}
Let $\J\subset A(\bD)$ be a non-trivial closed ideal. Then there exists an inner function $\omega\in H^\infty(\bD)$ and a closed subset $K\subset \bS_1$ of Lebesgue measure $0$ containing $X_{\omega}$ such that
\[
\J = \omega H^\infty(\bD)\cap \I(K) 
\]
where $\I(K)\subset A(\bD)$ is the ideal of functions vanishing on $K$.
\end{theorem}

The general case of the ball algebra was handled much later by Hedenmalm \cite{Hed89}. In this higher dimensional setting, the smaller supply of inner functions available makes the description less explicit. Given a set $\S\subset A(\bB_d)$, we let $Z(\S)\subset \ol{\bB_d}$ be the common zero set of the functions in $\S$; while given a set $K\subset \ol{\bB_d}$, we let $\I(K)\subset A(\bB_d)$ be the ideal of functions vanishing on $K$, as above. The description of ideals of $A(\bB_d)$ reads as follows.

\begin{theorem}[Hedenmalm]\label{T:hedenmalm}
Let $\J\subset A(\bB_d)$ be a closed ideal, and let $K = Z(\J)\cap \bS_d$. Then
\[
 \J = \widetilde{\J} \cap \I(K),
\]
where $\widetilde{\J}$ denotes the weak-$*$ closure of $\J$ in $H^\infty(\bB_d)$.
\end{theorem}

The proof of this result uses duality arguments.  It depends on the finer structure of  the dual space of $\AB$ which in turn hinges on results of Henkin \cite{Henkin}, Valskii \cite{Valskii} and Cole and Range \cite{ColeRange} (for a full treatment, see \cite[Chapter 9]{Rudin}). Note that the Carleson-Rudin theorem contains the information that the zero set of a proper ideal intersects $\bS_1$ in a closed set of Lebesgue measure zero, while  Hedenmalm's theorem contains almost no information about  zero sets.

In the present paper, we study the ideals and zero sets for the norm closed algebra $\A_d$ generated by the polynomial multipliers on the Drury-Arveson space.
Analogues of the aforementioned results on the dual space of $\AB$ were established by the authors in \cite{CD}, and they play an equally important role in our study of $\A_d$. 


The plan of the paper is as follows. In Section \ref{S:prelim} we introduce the necessary background and preliminaries on the Drury-Arveson space and the algebra $\A_d$. 

In Section \ref{S:dualityideals}, we establish several technical duality results that we require throughout the paper. We exhibit a intimate relation between the property of the boundary portion $K$ of a zero set being small with respect to the algebra $\A_d$, and the weak-$*$ closure in the dual space $\A_d^*$ of a certain space of measures on $K$. More precisely, we prove the following (see Corollary \ref{C:TNw*} and Theorem \ref{T:TSw*closed}). 

\begin{theorem}\label{T:intro1}
Let $\J\subset \A_d$ be a closed ideal, and let $K=Z(\J)\cap \bS_d$. Then, the space of $\A_d$--totally singular measures concentrated on $K$ is  closed in the weak-$*$ topology of $\A_d^*$ if and only if $K$ is $\A_d$-totally null.
\end{theorem}
 
Shifting our focus towards functions vanishing on a set rather than measures concentrated on it, we give the following characterization of a subset of the sphere being small with respect to $\A_d$ (see Corollaries \ref{C:I(K) dense} and \ref{C:I(K) dense ball}).

\begin{theorem}\label{T:intro2}
Let $K\subset \bS_d$ be a closed set. 
Then, $K$ is $\A_d$-totally null if and only if the ideal $\I(K)$ of functions in $\A_d$ vanishing on $K$ is weak-$*$ dense in the full multiplier algebra $\M_d$. When this happens, the unit ball of $\I(K)$ is weak-$*$ dense in the unit ball of $\M_d$.
\end{theorem}

Interpolation results (Corollary \ref{C:interpolation}) are obtained as additional byproducts of our duality tools.

In Section \ref{S:ideals}, we establish our structure theorem (Theorem \ref{T: ideals of A_d}) for ideals of $\A_d$, which has a striking resemblance with the classical Theorems \ref{T:carlesonrudin} and \ref{T:hedenmalm}. In particular, the weak-$*$ closure of an ideal plays a role. We also show in Theorem \ref{T:XminusY} that the zero set of an ideal and that of its weak-$*$ closure are the same up to a small set. 

\begin{theorem}\label{T:intro3}
Let $\J\subset \A_d$ be a closed ideal and let $\widetilde{\J}$ be its weak-$*$ closure in the full multiplier algebra $\M_d$. Let $K = Z(\J)\cap \bS_d$. Then, $\J = \I(K) \cap \widetilde\J$. Moreover, $Z(\J)\setminus Z(\widetilde{\J}\cap \A_d)$ is an $\A_d$-totally null subset of $\bS_d$.
\end{theorem}

Finally, in Section \ref{S:zerosets} we undertake a more exhaustive study of zero sets for $\A_d$. We show that closed $\A_d$--totally null subsets of the sphere are necessarily zero sets (Proposition \ref{P:Z(f)}), based on a sophisticated interpolation result from \cite{CD}. In addition, we provide supporting evidence that the converse may also be valid (Theorem \ref{T:zero + TN}), thus connecting our work with an old unresolved question of Rudin which asks whether the zero set of an ideal of $\AB$ is the zero set of a single function. Interestingly, our Theorems \ref{T:intro1}, \ref{T:intro2} and  \ref{T:intro3} provide a novel approach for a possible future resolution of this question. We also consider the problem of describing the smallest zero set for $\A_d$ containing a given subset $X$, and obtain rather definitive results in the case where $X$ is the intersection of a zero set with the sphere. Indeed, we show in Theorem \ref{T:smallestzeroset} that the following holds.

\begin{thm}\label{T:intro4}
Let $X$ be a zero set for $\A_d$ and let $X_0=X\cap\bS_d$. 
Define $\widehat{X_0}$ to be the set consisting of $X_0$ together with all points in $\bB_d$ 
which have a representing measure supported on $X_0$. 
Then, $\widehat{X_0}$ is the smallest zero set for $\A_d$ containing $X_0$, 
and  $X\setminus \widehat{X_0}$ is a countable discrete set.
\end{thm}

In fact, our proof applies to the uniform algebra setting as well. This result has a nice interpretation in complex geometry: up to a countable discrete set, an analytic variety in the ball is completely determined by its intersection with the boundary. 
This says in particular that if $V$ is a connected analytic variety which intersects the open ball,
then $V \cap \bB_d$ is determined by $V \cap \bS_d$. We have not been able to find a reference containing this result, even though it seems to answer a natural and classical question.

Lastly, we show that under some mild natural conditions, an interpolating sequence can be adjoined to an $\A_d$--totally null subset to obtain another zero set (Corollary \ref{C:zero interp}). This is meaningful from the operator theoretic point of view, as shown in \cite{abscont}.

\textbf{Acknowledgements.} The second author would like to thank Alex Izzo for several stimulating
conversations related to this paper.

\section{Background and preliminaries}\label{S:prelim}

Let $d\geq 1$ be an integer. The \emph{Drury-Arveson space}, denoted by $H_d^2$, is the reproducing kernel Hilbert space on the open unit ball $\bB_d\subset \bC^d$ with kernel given by
\[ k(z,w) = \frac1{1-\ip{z,w}} \qfor z,w \in \bB_d .\]
This space has the complete Nevanlinna-Pick property (\cite{quiggin93}), and furthermore it plays a central role in modern multivariate operator theory (\cite{Drury78,MV93,Arv98}). It is only a supporting character in our study however, as we are mostly interested in its  \emph{multiplier algebra}, which we denote by $\M_d$. Recall that a function $\phi:\bB_d\to \bC$ is a \emph{multiplier} if $\phi H^2_d\subset H^2_d$. The associated multiplication operator $M_\phi$ on $H^2_d$ is bounded, and we define the multiplier norm as follows
\[
\|\phi\|_{\M_d}=\|M_{\phi}\|.
\]
It is basic fact that
\[
\|\phi\|_{\infty}\leq \|\phi\|_{\M_d}
\]
for every $\phi\in \M_d$. While the polynomials are contained in $\M_d$, the inclusion 
\[
\M_d\subset H^\infty(\bB_d)
\]
is strict, and in fact the multiplier norm and supremum norm are not comparable. Throughout the paper we will focus on the norm closed subalgebra $\A_d$ of $\M_d$ generated by the polynomial multipliers. In particular, we have
\[
\A_d\subset \AB.
\]

By identifying a multiplier $\phi\in \M_d$ with the multiplication operator $M_\phi\in B(H^2_d)$, we may view $\M_d$ as an operator algebra on $H^2_d$, and this algebra is closed in the weak-$*$ topology of $B(H^2_d)$. 
In particular, $\M_d$ has a weak-$*$ topology. 
For sequences, this topology can easily be understood as follows:  $\{\phi_n\}_n\subset \M_d$ converges to $\psi\in \M_d$ in the weak-$*$ topology if and only if $\{\phi_n\}_n$ is bounded in the multiplier norm and converges pointwise to $\psi$ on $\bB_d$.

As mentioned in the introduction, several of our results have proofs that are based on duality arguments in $\A_d$. 
Because of the inequality
\[
\|\phi\|_{\infty}\leq \|\phi\|_{\M_d}, \quad \phi\in \A_d,
\]
every regular Borel measure on the unit sphere $\bS_d$ gives rise to a continuous linear functional on $\A_d$ via integration. 
Denote by $M(\bS_d)$ the space of such measures. 
Let $\mu\in M(\bS_d)$ and let $\Phi_\mu\in \A_d^*$ be the induced integration functional on $\A_d$.
We say that $\mu$ is \emph{$\A_d$-Henkin} if $\Phi_\mu$ extends to a weak-$*$ continuous linear functional on $\M_d$. 
At the other extreme, we say that $\mu$ is \emph{$\A_d$-totally singular} if it is singular with respect to every $\A_d$-Henkin measure. 
These definitions are direct analogues of the the classical ones for the ball algebra \cite[Chapter 9]{Rudin}. The point mass $\delta_\zeta$ is $\A_d$-totally singular for every $\zeta\in \bS_d$ by \cite[Proposition 6.1]{CD}. Given a Borel subset $E\subset \bS_d$, we let $\TS_{\A_d}(E)$ denote the space of all $\A_d$-totally singular measures $\mu$ which are concentrated on $E$, in the sense that $\mu(A\cap E)=\mu(A)$ for every Borel set $A$.

A Borel subset $K\subset \bS_d$ is called \emph{$\A_d$-totally null} if $|\mu|(K)=0$ for every $\A_d$-Henkin measure $\mu$. Equivalently, $K$ is $\A_d$-totally null if every measure supported on it is $\A_d$-totally singular (\cite[Corollary 5.6]{CD}). On such sets, we can achieve ``peak interpolation" using functions in $\A_d$ with
additional control on the multiplier norm of the interpolating function. The following deep result (\cite[Theorem 9.5]{CD}) is analogous to a classical theorem of Bishop  for uniform algebras \cite{Bishop}.

\begin{thm}\label{T:peakinterpolation}
Let $K\subset \bS_d$ be a closed $\A_d$-totally null subset and let $\ep>0$.
Then for every $f\in \rC(K)$, there exists $\phi\in \A_d$ such that 
\begin{enumerate}[label=\normalfont{(\roman*)}]
\item  $\phi|_K=f$,
\item  $|\phi(\zeta)|<\|f\|_K$ for every $\zeta\in \ol{\bB_d} \setminus K$, and
\item  $\|\phi\|_{\M_d}\leq (1+\ep)\|f\|_K $.
\end{enumerate}
\end{thm}

It is not known whether $\eps$ can be chosen to be zero in the statement above.

Note that since the multiplier norm and supremum norm are not comparable, not all continuous linear functionals on $\A_d$ are of the form $\Phi_{\mu}$ for some $\mu\in M(\bS_d)$.
The following result is a summary of \cite[Corollary 4.3, Theorem 4.4]{CD} and it paints a more complete picture.

\begin{thm}\label{T:dual}
The dual space $\A_d^*$ can be completely isometrically identified with $\M_{d*}\oplus_1 \TS_{\A_d}(\bS_d)$. In particular, any functional $\Psi\in \A_d^*$ can be decomposed as $\Psi=\Psi_a+\Phi_{\mu_s}$ where $\Psi_a\in \M_{d*}$,\ $\mu_s\in \TS_{\A_d}(\bS_d)$ and 
\[
\|\Psi\|_{\A_d^*}=\|\Psi_a\|_{\M_{d*}}+\|\mu_s\|_{M(\bS_d)}.
\]
\end{thm}

Here, $\M_{d*}$ denotes the standard predual of $\M_d$, namely 
\[
\M_{d*}=B(H^2_d)_*/\M_{d\perp}.
\]

Because of Theorem \ref{T:dual}, to understand $\A_d^*$ it is sufficient to understand both pieces $\M_{d*}$ and $\TS_{\A_d}(\bS_d)$ separately. One property of these spaces which we will require is that they are closed under absolute continuity: they form complementary \emph{bands}.

\begin{theorem}\label{T:band}
Let $\mu,\nu\in M(\bS_d)$ be such that $\mu$ is absolutely continuous with respect to $\nu$. If $\nu$ is $\A_d$-totally singular then so is $\mu$. If $\nu$ is $\A_d$-Henkin then so is $\mu$. Moreover every measure on $\bS_d$ decomposes uniquely as the sum of an $\A_d$-Henkin measure and an $\A_d$-totally singular measure.
\end{theorem}

Let us now briefly discuss the space $\M_{d*}$, which is more difficult to grasp than $\TS_{\A_d}(\bS_d)$. By definition, it contains $\Phi_{\mu}$ whenever $\mu$ is an $\A_d$-Henkin measure. However, since the norm on $\A_d$ is not comparable to the supremum norm over the ball, $\A_d^*$ must contain functionals which are not given as integration against some measure on the sphere. For every element $\Psi\in \M_{d*}$ and every $\eps>0$, there are functions $f,g\in H^2_d$ such that
\[ \Psi(\phi) = \langle M_\phi f,g \rangle_{H^2_d}, \quad \phi\in \A_d \]
and 
\[
\|f\|_{H^2_d} \, \|g\|_{H^2_d}<\|\Psi\|_{\A_d^*}+\eps.
\]
These facts follow from \cite[Theorem 2.10]{DP99}, but we will not require them explicitly.

Throughout the paper, if $X$ is a Banach space and $E\subset X$, we put
\[
E^\perp=\{\Lambda\in X^*:\Lambda(x)=0 \qforal x\in E\}
\]
while if $F\subset X^*$ we put
\[
F_\perp=\{x\in X: \Lambda(x)=0 \qforal \Lambda\in F\}.
\]
Standard separation arguments using the Hahn-Banach theorem show that if $E\subset X$ is a norm closed subspace and $F\subset X^*$ is weak-$*$ closed subspace, then
\[
(E^\perp)_\perp=E \qand (F_\perp)^\perp=F.
\]
For our purposes, we will need the following F.\&M.\ Riesz type result, which is \cite[Theorem 4.7]{CD}.
If $E\subset \A_d$, we let $\widetilde{E}\subset \M_d$ denote the closure of $E$ in the weak-$*$ topology of $\M_d$.
\begin{theorem}\label{T:dualJ}
Let $\J\subset \A_d$ be a closed ideal and let $\Psi\in \J^\perp$. Suppose that $\Psi=\Psi_a+\Phi_{\mu_s}$ where $\Psi_a\in \M_{d*}$ and $\mu_s\in \TS_{\A_d}(\bS_d)$. Then, $\Psi_a\in \widetilde{\J}_\perp$ and $\mu_s\in \TS_{\A_d}(Z(\J)\cap \bS_d)$, where $Z(\J)$ is the common zero set of the functions in $\J$.
\end{theorem}

\section{Duality and ideals in $\A_d$}\label{S:dualityideals}

In this section we prove duality results relating ideals and $\A_d$-totally singular measures that will be used throughout the paper. First, let us establish some terminology and notation.

If $\S\subset \A_d$ is a set, we let $Z(\S)\subset \ol{\bB_d}$ denote the common zero set of the functions in $\S$. Observe that if $\langle \S\rangle \subset \A_d$ denotes the closed ideal generated by $S$, then $Z(\S)=Z(\langle \S\rangle)$. Furthermore, if $X\subset\ol{\bB_d}$ is a set, we let $\I(X)\subset \A_d$ denote the closed ideal of functions vanishing on $X$.

It is easily verified that for any sets $\S\subset \A_d$ and $X\subset \ol{\bB_d}$, we have
\[
X\subset Z(\I(X)) \qand \langle \S \rangle\subset \I(Z(\S)).
\]
In particular, we have that
\begin{equation}\label{E:ZI}
Z(\I(Z(\S)))=Z(\S) \qand \I(Z(\I(X)))=\I(X).
\end{equation}

We start with a basic observation.

\begin{proposition}\label{P:TSpreann}
Let $E\subset \bS_d$ be a Borel set. Then
\[
\TS_{\A_d}(E)_\perp=\I(E).
\]
\end{proposition}

\begin{proof}
Obviously we have $\I(E)\subset \TS_{\A_d}(E)_\perp$. Conversely, let $\zeta\in E$. Then $\delta_{\zeta}\in \TS_{\A_d}(E)$ so that $\phi\in \TS_{\A_d}(E)_\perp$ implies $\phi(\zeta)=0$. Hence 
\[
\TS_{\A_d}(E)_\perp\subset \I(E) .\qedhere
\]
\end{proof}

The following result complements the previous proposition. The proof is inspired by a clever argument from \cite{Hed89}.

\begin{theorem}\label{T:TS annihilators}
Let $\J\subset \A_d$ be a closed ideal. Then 
\[ \J^\perp \cap  \TS_{\A_d}(\bS_d) \ =\  \I(Z(\J))^\perp \cap \TS_{\A_d}(\bS_d) \ =\  \TS_{\A_d}(Z(\J) \cap \bS_d) .\]
\end{theorem}

\begin{proof}
It is clear that
\[ \TS_{\A_d}(Z(\J) \cap \bS_d) \ \subset\  \I(Z(\J))^\perp \cap \TS_{\A_d}(\bS_d)  \ \subset\   \J^\perp \cap \TS_{\A_d}(\bS_d).\]
Take $\mu\in \J^\perp \cap \TS_{\A_d}(\bS_d)$.
Note then that $\phi\mu$ annihilates $\A_d$ for all $\phi\in \J$. Hence, $\phi\mu$ is trivially an $\A_d$-Henkin measure. On the other hand, $\phi\mu$ belongs to $\TS_{\A_d}(\bS_d)$ by Theorem \ref{T:band}; so that $\phi\mu=0$.
This shows that the support of $\mu$ is contained in the zero set of $\phi$.
Repeating this argument for every $\phi\in\J$ shows that $\mu$ is supported on $Z(\J)\cap \bS_d$, whence $\mu$ belongs to $\TS_{\A_d}(Z(\J) \cap \bS_d)$.
\end{proof}

Next, we explore some consequences of the previous observations.

\begin{corollary}\label{C:TNw*}
Let $\J\subset \A_d$ be a closed ideal, and let $K = Z(\J) \cap \bS_d$.  
If $\TS_{\A_d}(K)$ is closed in the weak-$*$ topology of $\A_d^*$, then $K$ is $\A_d$-totally null.
\end{corollary}

\begin{proof}
We have that
\[
\TS_{\A_d}(K)=(\TS_{\A_d}(K)_\perp)^\perp 
\]
whence Proposition \ref{P:TSpreann} implies that
\[
\I(K)^\perp\subset \TS_{\A_d}(\bS_d).
\]
Any $\A_d$-Henkin measure supported on $K$ lies in $\I(K)^\perp$, and by the previous inclusion it must be the zero measure. 
We conclude that $K$ is $\A_d$-totally null.
\end{proof}

In particular, taking $\J=\{0\}$ shows that $\TS_{\A_d}(\bS_d)$ is not closed in the weak-$*$ topology of $\A_d^*$ (in fact, it is weak-$*$ dense as was observed in \cite[Proposition 6.3]{CD}). 
The general problem of determining when $Z(\J)\cap \bS_d$ is $\A_d$-totally null is a difficult one and we return to it in Section \ref{S:zerosets}. 
For now, we show that the converse of Corollary \ref{C:TNw*} holds as well. In fact, the statement is true for general $\A_d$-totally null closed subsets.

\begin{theorem}\label{T:TSw*closed}
Let $\J\subset \A_d$ be a closed ideal. Let $E\subset \bS_d$ be an $\A_d$-totally null Borel set disjoint from $Z(\J)$ such that $\ol{E} \subset Z(\J) \cup E$. Then,
$
\J^\perp+\TS_{\A_d}(E)
$
is weak-$*$ closed.
\end{theorem}

\begin{proof}
Put $X=Z(\J)$ and $X_0 = X \cap \bS_d$. 
By the Krein-Smulyan theorem, it suffices to show that for every net $\{\Psi_{\alpha}\}_{\alpha}\subset \J^\perp+\TS_{\A_d}(E)$ with $\|\Psi_{\alpha}\|\leq 1$ which converges to $\Psi\in \A_d^*$ in the weak-$*$ topology, we have that 
$\Psi$ belongs to $\J^\perp+\TS_{\A_d}(E)$.
Now, for each $\alpha$ we can write 
\[
\Psi_\alpha = \Theta_\alpha + \Phi_{\mu_\alpha}
\]
for some $\Theta_{\alpha}\in \J^\perp$ and $\mu_\alpha\in \TS_{\A_d}(E)$.
Furthemore, by Theorem \ref{T:dualJ} we have
\[
\J^\perp = \widetilde\J_\perp \oplus_1 \TS_{\A_d}(X_0) ;
\]
so that for every $\alpha$ we may write
\[
\Theta_{\alpha}=\Lambda_{\alpha}\oplus \Phi_{\nu_{\alpha}}
\] 
for some $\Lambda_{\alpha}\in \widetilde{\J}_\perp$ and $\nu_\alpha\in \TS_{\A_d}(X_0)$.
Thus,
\[
\Psi_\alpha = \Lambda_\alpha \oplus \Phi_{\mu_\alpha + \nu_\alpha}.
\]
Since the measures $\mu_{\alpha}$ and $\nu_{\alpha}$ are concentrated on disjoint Borel sets, we have that
\[
\|\mu_{\alpha}\|\leq \|\mu_{\alpha}+\nu_\alpha\|\leq \|\Psi_{\alpha}\|\leq 1.
\]
Therefore, upon passing to a subnet we may suppose that $\{\mu_\alpha\}_\alpha$ converges to a measure $\mu$ concentrated on $\ol{E}\subset X\cup E$ in the weak-$*$ topology of $M(\bS_d)$.
Hence, $\{\Theta_\alpha\}_\alpha$ converges to $\Psi-\Phi_\mu$ in the weak-$*$ topology of $\A_d^*$. Since $\J^\perp\subset \A_d^*$  is weak-$*$ closed we see that $\Psi-\Phi_\mu\in \J^\perp$. Finally, note that we can decompose $\mu = \mu_1 + \mu_2$ where $\mu_1$ is concentrated on $E$ and $\mu_2$ is concentrated on $X_0$.
Since $E$ is assumed to be $\A_d$-totally null, we have $\mu_1\in \TS_{\A_d}(E)$; while $\Phi_{\mu_2} \in \J^\perp$ since $X_0\subset X=Z(\J)$. 
Hence
\[
\Psi-\Phi_{\mu}+\Phi_{\mu_2}\in \J^\perp
\]
and
\[
\Psi = (\Psi - \Phi_\mu + \Phi_{\mu_2}) + \Phi_{\mu_1}  \in \J^\perp+\TS_{\A_d}(E) .
\]
Thus $\J^\perp+\TS_{\A_d}(E)$ is closed in the weak-$*$ topology of $\A_d^*$.
\end{proof}

Note that upon taking $\J=\A_d$ in the theorem above, we see that $\TS_{\A_d}(E)$ is closed in the weak-$*$ topology of $\A_d^*$ whenever $E\subset \bS_d$ is a closed $\A_d$-totally null subset. We obtain the following density result as an immediate consequence.

\begin{cor}\label{C:I(K) dense}
Let $K\subset \bS_d$ be a closed set. 
Then, $K$ is $\A_d$-totally null if and only if $\I(K)$ is weak-$*$ dense in $\M_d$.
\end{cor}

\begin{proof}
Let $\J$ be the weak-$*$ closure of $\I(K)$ in $\M_d$ and note that 
\[
\J_\perp =\M_{d*} \cap  \I(K)^\perp.
\]

Assume first that $\I(K)$ is weak-$*$ dense in $\M_d$, that is $\J=\M_d$. Any $\A_d$-Henkin measure supported on $K$ lies in $\I(K)^\perp\cap \M_{d*}=\J_\perp=\{0\}$; so that $K$ is $\A_d$-totally null. 

Assume conversely that $K$ is $\A_d$-totally null.
Now, we have that 
\[
\TS_{\A_d}(K)_\perp=\I(K)
\]
by Proposition \ref{P:TSpreann}. By Theorem \ref{T:TSw*closed} we see that
\[
\I(K)^\perp=(\TS_{\A_d}(K)_\perp)^\perp=\TS_{\A_d}(K)
\]
is disjoint from $\M_{d*}$; whence $\J_\perp=\{0\}$ and $\J = (\J_\perp)^\perp = \M_d$.
\end{proof}

In fact a stronger density statement holds.

\begin{cor}\label{C:I(K) dense ball}
Let $K\subset \bS_d$ be a closed $\A_d$-totally null subset. 
Then the unit ball of $\I(K)$ is weak-$*$ dense in the unit ball of $\M_d$.
\end{cor}

\begin{proof}
As observed in the previous proof, we have
\[
\I(K)^\perp =\TS_{\A_d}(K) \subset \TS_{\A_d}(\bS_d).
\]
Therefore, by Theorem \ref{T:dual} we conclude that
\[
 \I(K)^* = \A_d^*/\I(K)^\perp = \M_{d*} \oplus_1 (\TS_{\A_d}(\bS_d)/\TS_{\A_d}(K)) .
\]
Consequently we obtain that
\[
 \I(K)^{**} =\M_d \oplus \TS_{\A_d}(K)^{\perp} .
\]
The desired result is now a consequence of Goldstine's theorem.
\end{proof}

We close this section with another interesting density result that depends on duality techniques. The other tool we need for its proof is the following, which is of independent interest. Given $\lambda\in \ol{\bB_d}$, we denote by $\tau_{\lambda}\in \A_d^*$ the functional of evaluation at $\lambda$.

\begin{theorem}\label{T:isometricquotient}
Let $K\subset \bS_d$ be a closed $\A_d$-totally null subset and let $\lambda\in \bB_d$. Then, 
\[
\I(K\cup \{\lambda\})^\perp=\TS_{\A_d}(K)\oplus_1\bC \tau_{\lambda}
\]
and there exists a surjective isometric isomorphism
\[
\rho: \A_d/\I(K\cup\{\lambda\}) \to \rC(K\cup\{\lambda\})
\]
given by the restriction map
\[
\rho(\phi+\I(K\cup\{\lambda\}))=\phi|_{K\cup\{\lambda\}} \qfor \phi\in \A_d .
\]
\end{theorem}

\begin{proof}
By Proposition \ref{P:TSpreann}, we have that
\[
\TS_{\A_d}(K)_\perp=\I(K).
\] 
Invoking Theorem \ref{T:TSw*closed}, we see that $\TS_{\A_d}(K)$ is closed in the weak-$*$ topology of $\A_d^*$ and thus
\[
\TS_{\A_d}(K)+\bC \tau_{\lambda}
\]
is closed in the weak-$*$ topology of $\A_d^*$. Moreover,
\begin{align*}
(\TS_{\A_d}(K)+\bC \tau_{\lambda})_\perp&=\TS_{\A_d}(K)_\perp\cap (\bC \tau_{\lambda})_\perp\\
&=\I(K)\cap \I(\{\lambda\})\\
&=\I(K\cup \{\lambda\}).
\end{align*}
 Hence
\[
\I(K\cup \{\lambda\})^\perp=\TS_{\A_d}(K)+\bC \tau_{\lambda}=\TS_{\A_d}(K)\oplus_1 \bC \tau_{\lambda}
\]
where the last equality follows from Theorem \ref{T:dual}, since  $\lambda\in \bB_d$ and therefore $\tau_\lambda\in \M_{d*}$. The map $\rho$ is clearly a well-defined homomorphism, so we need only show that it is isometric and surjective.

Recall now that $\delta_{\zeta}\in \TS_{\A_d}(K)$ for every $\zeta\in K$, so that
\[
\sup_{z\in K\cup \{\lambda\}}|\phi(z)|\leq \sup\left\{ |\Psi(\phi)|:\Psi\in \I(K\cup \{\lambda\})^\perp, \|\Psi\|\leq 1\right\}
\]
whence
\[
\sup_{z\in K\cup \{\lambda\}}|\phi(z)|\leq \|\phi+\I(K\cup \{\lambda\})\|_{\A_d/\I(K\cup \{\lambda\}}
\]
for every $\phi\in \A_d$.

Conversely, if $\Psi\in \I(K\cup \{\lambda\})^\perp$ with $\|\Psi\|\leq 1$, we can write
$\Psi=c\tau_{\lambda}+\Phi_{\mu}$ with $c\in \bC, \mu\in \TS_{\A_d}(K)$ and 
\[
 1 \ge \|c\tau_{\lambda}+\Phi_{\mu}\|=\|c\tau_{\lambda}\|+\|\Phi_{\mu}\|=|c|+\|\mu\| .
\]
Hence,
\[
|\Psi(\phi)|\leq |c||\phi(\lambda)|+\|\mu\| \sup_{z\in K}|\phi(z)|\leq \sup_{z\in K\cup\{\lambda\}}|\phi(z)| .
\]
Therefore
\[
\|\phi+\I(K\cup\{\lambda\})\|\leq \sup_{z\in K\cup\{\lambda\}}|\phi(z)|
\]
for every $\phi\in \A_d$, and equality is established; so the map $\rho$ is isometric.

To see that $\rho$ is surjective, observe that the dual space of $\A_d/\I(K\cup\{\lambda\})$ is naturally identified with 
\[\I(K\cup \{\lambda\})^\perp=\TS_{\A_d}(K)\oplus_1\bC \tau_{\lambda} .\]
However since $K$ is $\A_d$-totally null, $M(K) = \TS_{\A_d}(K)$ and thus 
\[ \A_d/\I(K\cup\{\lambda\})^* = M(K \cup\{\lambda\}) = \rC(K \cup\{\lambda\})^* .\]
It follows from the Hahn-Banach theorem that 
\[ \rho(\A_d/\I(K\cup\{\lambda\})) = \rC(K\cup\{\lambda\})  .\qedhere \]
\end{proof}

As a consequence, we obtain the following interpolation result for functions in $\A_d$.

\begin{cor}\label{C:interpolation}
Let  $K\subset \bS_d$ be a closed $\A_d$-totally null subset and $\lambda \in \bB_d$. 
Then the restriction of the open unit ball of $\A_d$ to $K \cup \{\lambda\}$ coincides with the open unit ball of $\rC(K \cup \{\lambda\})$.
\end{cor}

\begin{proof}
Let $q:\A_d\to\A_d/\I(K\cup\{\lambda\})$ be the quotient map. 
By Theorem \ref{T:isometricquotient},  the restriction map
\[
\rho:\A_d/\I(K\cup\{\lambda\})\to \rC(K\cup\{\lambda\})
\]
is an isometric isomorphism.
In particular, the open unit ball of $\rho(q(\A_d))$ coincides with the open unit ball of $\rC(K \cup \{\lambda\})$.
Let $f\in \rC(K\cup\{\lambda\})$ such that $\|f\|<1$. 
Then, there is $\phi\in \A_d$ such that $\phi|_{K\cup \{\lambda\}}=f$ and $\|q(\phi)\|=\|f\|$. 
By definition of the quotient norm we can find $\psi\in \A_d$ such that $\|\psi\|<1$ and $\psi|_{K\cup\{\lambda\}}=f$, which completes the proof.
\end{proof}

In Section 5, we obtain a similar result for interpolating sequences (Theorem~\ref{T:interp}).

It would be interesting to know if a full analogue of Theorem \ref{T:peakinterpolation} holds for $K\cup\{\lambda\}$. 
Furthermore, it would be of interest to understand what happens when $\{\lambda\}$ is replaced by an arbitrary finite set $F\subset \bB_d$. 
The restriction of $\A_d$ (or even $\AB$) to a two point set $F = \{\lambda,\mu\}\subset \bB_d$ cannot be isometric to $\rC(F)$ because of Schwarz's lemma. The complete Nevanlinna-Pick property governs the norm of a multiplier that achieves specific values at these two points.

\section{Description of the ideals of $\A_d$}\label{S:ideals}

In this short section, we provide the structure theorem for ideals of $\A_d$. 
This result is a natural analogue of Theorems \ref{T:carlesonrudin} and \ref{T:hedenmalm}. 
Recall that if $\J\subset \A_d$, then $\widetilde{\J}\subset \M_d$ denotes its closure in the weak-$*$ topology of $\M_d$.

\begin{thm}\label{T: ideals of A_d}
Let $\J\subset \A_d$ be a closed ideal, and let $K = Z(\J)\cap \bS_d$.
Then 
\[ \J = \I(K) \cap \widetilde\J .\]
\end{thm}

\begin{proof}
By Theorem \ref{T:dualJ}, we have that 
\[
\J^\perp = \widetilde{\J}_\perp \oplus_1 \TS_{\A_d}(K).
\]
We see that
\[
(\widetilde{\J}_{\perp})_\perp=\{\phi\in \A_d: \Psi(\phi)=0 \qforal \Psi\in \widetilde{\J}_\perp\}=\widetilde{\J}\cap \A_d.
\]
Moreover,
\[
\I(K)= \TS_{\A_d}(K)_\perp
\]
by Proposition \ref{P:TSpreann}.
We conclude that
\begin{align*}
 \J &= (\J^\perp)_\perp = (\widetilde{\J}_\perp)_\perp\cap \TS_{\A_d}(K)_\perp \\
 &= (\widetilde\J \cap \A_d) \cap \I(K)\\
 &=\widetilde{\J}\cap \I(K). \qedhere
\end{align*}
\end{proof}

As a consequence, we obtain the following.

{\samepage 
\begin{corollary}\label{C:Jtilde}
Let $\J\subset \A_d$ be a closed ideal, and let $\widetilde{\J}\subset \M_d$ be its weak-$*$ closure. 
Then, the following statements are equivalent:
\begin{enumerate}
\item[\rm{(i)}] $\widetilde{\J}\cap \A_d=\J$,

\item[\rm{(ii)}] $Z(\widetilde{\J}\cap \A_d)=Z(\J)$, and

\item[\rm{(iii)}] $Z(\widetilde{\J}\cap \A_d)\cap \bS_d=Z(\J)\cap \bS_d$.
\end{enumerate}
\end{corollary}
}  

\begin{proof}
The only non-trivial implication is that (iii) implies (i). Assume that 
\[
Z(\widetilde{\J}\cap \A_d)\cap \bS_d=Z(\J)\cap \bS_d
\]
and denote this set by $X$. Note  that
\[
\J\subset \widetilde{\J}\cap \A_d\subset \widetilde{\J};
\] 
whence $\widetilde{\J}$ is the weak-$*$ closure of $\widetilde{\J}\cap \A_d$. 
By Theorem \ref{T: ideals of A_d}, we may write
\[
\J=\widetilde{\J}\cap \I(X)=\widetilde{(\widetilde{\J}\cap \A_d)} \cap \I(X) =  \widetilde{\J}\cap \A_d,
\]
and the proof is complete.
\end{proof}

Consequently we see that an equality between ideals can be detected by an equality between zero sets. 
Clearly, we have that $\J\subset \widetilde{\J}\cap \A_d$ and hence $Z(\widetilde{\J}\cap \A_d)\subset Z(\J)$.
Remarkably, the set $Z(\J)\setminus Z(\widetilde{\J}\cap \A_d)$ is always rather small and is contained in $\bS_d$.

\begin{theorem}\label{T:XminusY}
Let $\J\subset \A_d$ be a closed ideal and let $\widetilde{\J}\subset \M_d$ be its weak-$*$ closure. Then, $Z(\J)\setminus Z(\widetilde{\J}\cap \A_d)$ is an $\A_d$-totally null subset of $\bS_d$.
\end{theorem}

\begin{proof}
Let $X=Z(\J)$ and $Y=Z(\widetilde{\J}\cap \A_d)$. 
Put $K=X\setminus Y$. 
If $\phi$ is in $\widetilde{\J}$,  then there exists a sequence $\{\psi_n\}_n\subset \J$ which converges to $\phi$ in the weak-$*$ topology of $\M_d$.  In particular, since $\tau_\lambda$ is $\A_d$-Henkin for $\lambda\in\bB_d$, it follows that $\phi(\lambda)=0$ for every $\lambda\in Z(\J)\cap \bB_d$;
whence 
\[
Z(\J)\cap \bB_d=Z(\widetilde{\J}\cap \A_d)\cap \bB_d = Z(\widetilde{\J}) .
\]
We conclude that $K\subset \bS_d$.

Let $\eta$ be an $\A_d$-Henkin measure. We claim that $|\eta|(K)=0$. 
Note that the restriction of the measure $|\eta|$ to $K$ is also $\A_d$-Henkin 
because it is absolutely continuous with respect to $\eta$ (Theorem \ref{T:band}). 
Call this measure $\mu$. Let $\zeta_0\in K$ and 
choose a function $\phi\in \widetilde{\J}\cap \A_d$ such that $\phi(\zeta_0)\neq 0$. 
There exists a sequence $\{\psi_n\}_n\subset \J$ converging to $\phi$ in the weak-$*$ topology of $\M_d$. 
Observe that $\ol{\phi}\mu$ is $\A_d$-Henkin by Theorem \ref{T:band} so we find
\[
\int_{K} |\phi|^2 \,d\mu = \lim_{n\to \infty} \int_{K}\psi_n \ol{\phi} \,d\mu = 0
\]
since each $\psi_n$ vanishes on $X$. In particular, 
\[ \mu(\{\zeta\in K: \phi(\zeta)\neq 0\})=0 . \]
Since $\phi$ is continuous and $\phi(\zeta_0)\neq 0$, this shows that $\zeta_0$ does not lie in the support of $\mu$. 
As $\zeta_0\in K$ was arbitrary and $\eta$ is regular, we conclude that $|\eta|(K)=0$. Since $\eta$ was an arbitrary $\A_d$-Henkin measure,  $K$ is $\A_d$-totally null.
\end{proof}

\section{Zero sets for $\A_d$}\label{S:zerosets}

In this section we will be interested in \emph{zero sets} for $\A_d$, that is closed sets $X \subset \ol{\bB_d}$ such that there is a set $\S \subset \A_d$ for which $X=Z(\S)$. We focus in particular on the size of zero sets, and on the intersection of zero sets with the sphere.
In the case of the ball algebra, it is a classical fact that a closed subset of the sphere 
is the zero set of a single function in $\AB$ if and only if it is $\AB$--totally null (see \cite[Chapter 10]{Rudin} for the appropriate definition and statement).
In $\A_d$, we can establish one of these implications.

\begin{proposition}\label{P:Z(f)}
Let $K\subset \bS_d$ be a closed $\A_d$-totally null subset.
Then there is a function $\phi\in\A_d$ such that $Z(\phi)=K$.
\end{proposition}

\begin{proof}
By Theorem \ref{T:peakinterpolation}, there exists $\psi\in \A_d$ such that 
\[
 \psi|_K=1  \qand  |\psi(\zeta)|<1 \ \FOR \zeta\in \bS_d\setminus K . 
\]
Hence, we see that for $z\in \ol{\bB_d}$, we have $\psi(z)=1$  if and only if $z\in K$.
If we let $\phi = 1-\psi$, then $Z(\phi)=K$. 
\end{proof}

We do not know whether the converse holds. 

\begin{question}\label{Q:single}
Let $\phi\in \A_d$ such that $Z(\phi)$ is contained in $\bS_d$. Must $Z(\phi)$ be $\A_d$-totally null?
\end{question}
 
More generally, an old unresolved question of Rudin (\cite[page 415]{Rudin}) asks whether the zero set of an ideal of $\AB$ which is contained in the sphere is necessarily $\AB$--totally null (equivalently, the zero set of a \emph{single} function in $\AB$). One can formulate the corresponding problem in $\A_d$.

\begin{question}\label{Q:Rudin}
Let $X$ be a zero set for $\A_d$ that is contained in $\bS_d$. Must $X$ be $\A_d$-totally null?
\end{question}

Since Rudin's problem has been around for a long time, we expect this question to be very difficult to answer. 
Nevertheless, Corollaries \ref{C:TNw*}, \ref{C:I(K) dense} and Theorem \ref{T:TSw*closed} provide a different approach to the problem. 
Moreover, note that if $X\subset \bS_d$ is a zero set, then putting $\J=\I(X)$, we have $X=Z(\J)$ by (\ref{E:ZI}). Therefore
\begin{equation}\label{E:union}
X = \big( Z(\widetilde{\J}\cap \A_d)\cap \bS_d \big) \ \cup\ \big( Z(\J)\setminus Z(\widetilde{\J}\cap \A_d) \big)
\end{equation}
since 
\[
Z(\widetilde{\J}\cap \A_d)\cap \bB_d=Z(\J)\cap \bB_d
\]
as was shown in the proof of Theorem \ref{T:XminusY}. Note that the union in (\ref{E:union}) is disjoint, and that the second member in the right-hand side of the equality is $\A_d$-totally null by Theorem \ref{T:XminusY}. 
In particular, $X$ is $\A_d$-totally null if and only if $Z(\widetilde{\J}\cap \A_d)\cap \bS_d$ is. This observation may help in answering Question \ref{Q:Rudin}, as the ideal $\widetilde{\J}\cap \A_d$ may be easier to handle given our understanding of the weak-$*$ closed ideals of $\M_d$ 
(see \cite{DRS2,GRS,McCT00}).

Next, we show that certain (possibly non-closed) $\A_d$--totally null subsets of $\bS_d$ can be added to zero sets to obtain new zero sets.

\begin{theorem} \label{T:zero + TN}
Let $\J\subset \A_d$ be a closed ideal. Let $E\subset \bS_d$ be an $\A_d$-totally null Borel set such that $\ol{E} \subset Z(\J)\cup E$. If we put $\J'= \J\cap \I(E)$,
then 
\[ Z(\J') = Z(\J) \cup E . \]
\end{theorem}

\begin{proof}
Put $X=Z(\J)$. We may suppose that $E$ is disjoint from $X$ by replacing it with $E \setminus X$ if necessary. It is clear that $X\cup E\subset Z(\J')$, so we need only show the reverse inclusion.

Observe that
\[
( \J^\perp + \TS_{\A_d}(E) )_\perp = (\J^\perp)_\perp \cap \TS_{\A_d}(E)_\perp = \J \cap \I(E) = \J' 
\]
by Proposition \ref{P:TSpreann}.
By Theorem \ref{T:TSw*closed}, we see that
$\J^\perp + \TS_{\A_d}(E)$ is closed in the weak-$*$ topology of $\A_d^*$, whence
\[ \J'^\perp  =  \J^\perp + \TS_{\A_d}(E)= \widetilde\J_\perp \oplus_1 \left(\TS_{\A_d}(X \cap \bS_d)+\TS_{\A_d}(E)\right)\]
by Theorem \ref{T:dualJ}.
Let now $\lambda\in Z(\J')$. If $\lambda\in \bB_d$, then the functional $\tau_{\lambda}$ of evaluation at $\lambda$ lies in 
$\J'^\perp\cap \M_{d*}=\widetilde{\J}_\perp\subset \J^\perp.$ In particular, we must have $\lambda\in Z(\J)=X$. If $\lambda\in \bS_d$, then $\tau_\lambda$ is given by integration against $\delta_\lambda$, so that 
\[
\delta_{\lambda}\in \TS_{\A_d}(\bS_d)\cap \J'^\perp=\TS_{\A_d}(X \cap \bS_d)+\TS_{\A_d}(E) ;
\]
whence $\lambda\in (X\cap \bS_d)\cup E$. We conclude that
\[
Z(\J')\subset X\cup E .\qedhere
\]
\end{proof}

An immediate consequence of this result is the following.

\begin{cor}\label{C:zerosetunionTN}
If $X\subset \ol{\bB_d}$ is a zero set for $\A_d$ and $E$ is an $\A_d$-totally null Borel set such that $\ol{E}\subset X\cup E$, then $X\cup E$ is  a zero set for $\A_d$.
\end{cor}

In the special case where the zero set $X$ lies in the sphere, the conclusion of that corollary is consistent with a positive answer for Question \ref{Q:Rudin}, and thus can be seen as supporting evidence.

Upon examining Question \ref{Q:Rudin}, one may wonder to what extent a zero set for $\A_d$ is determined by its intersection with the sphere. In order to address this issue, we need a standard lemma.

\begin{lemma}\label{L:characters}
Let $X\subset \ol{\bB_d}$ be a closed subset. Then, the character space of $\A_d/\I(X)$ can be canonically identified with $Z(\I(X))$. Moreover, under this identification the Gelfand transform 
\[
\Gamma: \A_d/\I(X)\to \rC(Z(\I(X)))
\]
is given by the restriction map 
\[
\Gamma(\phi+\I(X))=\phi|_{Z(\I(X))} \qfor \phi\in \A_d .
\]
\end{lemma}

\begin{proof}
Let
\[
q:\A_d\to \A_d/\I(X)
\]
be the quotient map. 
As this map is a surjective homomorphism, the adjoint map $q^*$ is an injective map from $(\A_d/\I(X))^*$ to $\A_d^*$ which takes characters to characters.
Let $\chi$ be a character of $\A_d/\I(X)$. Then $q^*(\chi) = \chi\circ q$ is a character of $\A_d$.
Hence there is a $\lambda\in\ol{\bB_d}$ so that $\chi \circ q = \tau_\lambda$, where $\tau_{\lambda}$ is evaluation at $\lambda$.
Observe that 
\[
\phi(\lambda)=(\chi\circ q)(\phi)=0 \qforal \phi\in \I(X) ;
\]
whence $\lambda\in Z(\I(X))$.

Conversely, let $\lambda\in Z(\I(X))$. 
Then, $\I(X)\subset \ker \tau_{\lambda}$; so  there exists a character $\chi_{\lambda}$ of $\A_d/\I(X)$ with $\chi_\lambda\circ q=\tau_\lambda$. 
Hence the character space of $\A_d/\I(X)$ is canonically identified with $Z(\I(X))$. The claim about the Gelfand transform follows immediately.
\end{proof}

Recall that given a commutative Banach algebra $\A$, its \emph{Shilov boundary} is the smallest closed subset $\Sigma$ of the character space $\Delta_{\A}$ with the property that
\[
\max_{\chi\in \Delta_\A}|\chi(a)|=\max_{\chi\in \Sigma}|\chi(a)|
\]
for every $a\in \A$. The next lemma establishes a simple topological property of the Shilov boundary for some quotients of $\A_d$.

\begin{lemma}\label{L:shilov}
Let $X\subset\ol{\bB_d}$ and let $\Sigma_X$ denote the Shilov boundary of the algebra $\A_d/\I(X)$. 
Then, $\Sigma_X$ contains $Z(\I(X))\cap \bS_d$ 
and $\Sigma_X\cap \bB_d$ consists of all isolated points of $Z(\I(X))\cap \bB_d$.
Moreover, if $z$ is an isolated point of $Z(\I(X)) \cap \bB_d$, then there is a $\phi \in \A_d$ such that 
\[
 \phi(z)=1 \qand \phi|_{Z(I(X))\setminus\{z\}} = 0 .
\]
\end{lemma}

\begin{proof}
By virtue of Lemma \ref{L:characters} we see that the character space of $\A_d/\I(X)$ is $Z(\I(X))$ and that the Gelfand transform is the restriction of a function in $\A_d$ to that set.
For $\zeta\in Z(\I(X))\cap \bS_d$, let 
\[
\theta_{\zeta}(z)=\frac{1+\langle z,\zeta \rangle_{\bC^d}}{2}, \quad z\in \bB_d.
\]
Then, $\theta_\zeta\in \A_d$ and
\[
|\theta_\zeta(\lambda)|<|\theta_\zeta(\zeta)|=1
\]
whenever $\lambda\in \ol{\bB_d}\setminus\{\zeta\}$; whence $\zeta\in \Sigma_X$. Thus, $Z(\I(X))\cap \bS_d\subset \Sigma_X$.

Next, let $z$ be a non-isolated point of $Z(\I(X))\cap \bB_d$. Since $Z(\I(X))\cap \bB_d$ is a complex analytic variety, there is an irreducible subvariety $V$ containing $z$ which has dimension at least 1 (see \cite[Theorem 3.2B]{Whitney}). 
By the maximum modulus principle (see \cite[Theorem~III.B.16]{GR}), we see that $\ol{V}\cap \bS_d$ is non-empty and
\[
|\phi(z)| \leq \max_{\lambda\in \ol{V} \cap \bS_d} |\phi(\lambda)| 
\]
for every $\phi\in \A_d$. In particular, $\Sigma_X\subset \ol{\bB_d}\setminus (V\cap \bB_d)$ and $z$ does not belong to $\Sigma_X$.

Finally suppose that $z$ is an isolated point of $Z(\I(X))\cap \bB_d$.
By the Shilov idempotent theorem (see \cite[Corollary~III.6.53]{Gamelin}), there is an element $b \in \A_d/\I(X)$ such that 
\[
 \Gamma(b)(z)=1 \qand \Gamma(b) |_{Z(I(X))\setminus\{z\}} = 0 .
\]
In particular, there is $\phi\in \A_d$ such that 
\[
\phi(z)=1 \qand \phi |_{Z(I(X))\setminus\{z\}} = 0 .
\]
Therefore $z$ belongs to the Shilov boundary.
\end{proof}

Recall also that a \emph{representing measure} for a point $\lambda\in \ol{\bB_d}$ is a regular Borel measure $\mu$ on $\ol{\bB_d}$ satisfying
\[
\phi(\lambda)=\int \phi \,d\mu
\]
for every $\phi\in \A_d$. We show that for non-isolated points of a zero set inside the ball, there always exists a representing measure supported on the part of that zero set lying in the sphere.

\begin{lemma}\label{L:repmeasure}
Let $X\subset\ol{\bB_d}$ and let $\lambda\in Z(\I(X))\cap \bB_d$ be a non-isolated point. 
Then there is a positive representing measure for $\lambda$ that is supported on $Z(\I(X))\cap \bS_d$.
\end{lemma}

\begin{proof}
The closure of the Gelfand transform of $\A_d/\I(X)$ can be considered as a uniform algebra on its Shilov boundary $\Sigma_X$. Thus, 
there is a positive representing measure $\mu$ for the non-isolated point $\lambda\in Z(\I(X))\cap \bB_d$ supported on $\Sigma_X$. Let $z\in Z(\I(X))\cap \bB_d$ be an isolated point. In particular, $z\neq \lambda$.  By Lemma \ref{L:shilov}, there is $\phi\in \A_d$ such that $\phi(z)=1$ while $\phi$ vanishes on $Z(\I(X)) \setminus \{z\}$. We find
\[
0=\phi(\lambda)=\int_{\Sigma_X}\phi d\mu= \mu(\{z\})
\]
and thus $z$ does not belong to the support of $\mu$. By Lemma \ref{L:shilov} again, we conclude that $\mu$ must be supported on $Z(\I(X))\cap \bS_d$.
\end{proof}

It turns out that a zero set for $\A_d$ is completely determined by its intersection with the sphere, except for isolated points in the ball. This is the content of the next result. 

\begin{thm}\label{T:smallestzeroset}
Let $X$ be a zero set for $\A_d$ and let $X_0=X\cap\bS_d$. 
Define $\widehat{X_0}$ to be the set consisting of $X_0$ together with all points in $\bB_d$ 
which have a representing measure supported on $X_0$. 
Then, $\widehat{X_0}$ is the smallest zero set for $\A_d$ containing $X_0$, 
and  $X\setminus \widehat{X_0}$ is a countable discrete set.
\end{thm}

\begin{proof}
Let $Y$ be the smallest zero set for $\A_d$ containing $X_0$ and put $\B = \A_d/\I(Y)$. Note that $Y\cap \bS_d=X_0$ since $X$ is a zero set.
By Lemma \ref{L:characters}, the character space of $\B$ is canonically identified with 
\[
Z(\I(Y))=Y, 
\]
since $Y$ is a zero set (use (\ref{E:ZI})). 
Moreover, under this identification, the Gelfand transform $\Gamma: \B\to \rC(Y)$ is given as $\Gamma(q(\phi))=\phi|_Y$ for every $\phi\in \A_d$, where $q:\A_d\to \B$ is the quotient map. 

Let $\Sigma\subset Y$ be the Shilov boundary of $\B$. 
By Lemma~\ref{L:shilov}, $\Sigma$ consists of $X_0$ together with any isolated points of $Y\cap \bB_d$.
However if $z$ is an isolated point of $Y\cap \bB_d$, then the same lemma provides a function $\phi \in \A_d$ such that 
\[
 \phi(z)=1 \qand \phi |_{Y\setminus\{z\}} = 0 .
\]
Therefore $Y \cap Z(\phi) = Y\setminus\{z\}$ is a zero set containing $X_0$ that is smaller than $Y$.
This contradiction establishes that $Y$ has no isolated points in the ball.
We conclude that $X_0$ is the Shilov boundary of $\B$. 

In particular, the uniform closure $\C=\ol{\Gamma(\B)}\subset \rC(Y)$ is a uniform algebra with Shilov boundary given by $X_0$. Hence, $\C$ can be considered as a subalgebra of $\rC(X_0)$.
It follows that the evaluation functional at any point of $Y\cap\bB_d$ has a representing measure supported on $X_0$. 
We conclude that $Y\subset \widehat{X_0}$.

Note now that if $\lambda\in \ol{\bB_d}$ has a representing measure supported on $X_0$, then clearly any $\phi\in\A_d$ vanishing on $X_0$ must also satisfy $\phi(\lambda)=0$. Hence $\widehat{X_0}\subset Z(\I(X_0)).$
Since $Y$ is a zero set containing $X_0$, we also have that
\[
\widehat{X_0}\subset Z(\I(X_0))\subset Z(\I(Y))=Y,
\]
so in fact $Y=\widehat{X_0}$ (here again we used (\ref{E:ZI})). Finally, $X \setminus \widehat{X_0}$ must consist of isolated points by Lemma \ref{L:repmeasure}, and thus it is countable and discrete.
\end{proof}

We mention that the proof technique used above applies verbatim in the uniform algebra setting, and appears to be new there as well.
Let $\Omega$ be a bounded, strictly pseudo-convex domain in $\bC^n$ which is the interior of its closure.
Let $B =  \ol{\Omega} \setminus\Omega$.
Consider the Banach algebra $A(\Omega)$ consisting of functions in $\rC(\ol{\Omega})$ which are
analytic on $\Omega$. 
It is known that the character space of $A(\Omega)$ can be identified with $\ol{\Omega}$ via point evaluation.
We can prove the following:

\begin{thm}
Let $X$ be a zero set for $A(\Omega)$ and let $X_0=X\cap B$. 
Define $\widehat{X_0}$ to be the set consisting of $X_0$ together with all points in $\Omega$ 
which have a representing measure supported on $X_0$. 
Then, $\widehat{X_0}$ is the smallest zero set for $A(\Omega)$ containing $X_0$, 
and  $X\setminus \widehat{X_0}$ is a countable discrete set.
\end{thm}

Looking at the proof, one sees that non-trivial information from several complex variables was used when we apply the Shilov idempotent theorem.
As mentioned in the introduction, this result has implications for analytic varieties, and even algebraic varieties, that we have not found in the literature or by consulting experts in those areas.
In particular, if $V$ is a connected analytic variety that intersects the open ball $\bB_d$, then $V\cap\bB_d$ is determined by $V\cap \bS_d$.
We expected that this would be a classical observation, but so far  we have no reference for this fact  even in the case of an algebraic variety.

As another remark, note that one cannot replace $X_0$ in Theorem \ref{T:smallestzeroset} by an arbitrary closed subset of $\bS_d$å. For example, if $X_0\subset \bT$ is a proper subset with positive Lebesgue measure,
then it is a classical fact that the smallest zero set for $\A_1=\AD$ containing $X_0$ is $\ol{\bD}$, which differs from $ \widehat{X_0}$. It thus appears to be much more difficult to obtain a precise picture describing the smallest zero set for $\A_d$ containing an arbitrary subset $X\subset \bS_d$. One basic fact we can prove in general is the following.

\begin{proposition}\label{P:smallestgeneral}
Let $X\subset \ol{\bB_d}$. Then, $X$ intersects every connected component of $Z(\I(X))$.
\end{proposition}

\begin{proof}
By Lemma \ref{L:characters}, we see that the character space of $\A_d/\I(X)$ is $Z(\I(X))$ and that the Gelfand transform is simply the restriction of a function in $\A_d$ to that set.
Assume that there is a connected component $U$ of $Z(\I(X))$ which is disjoint from $X$. 
By the Shilov idempotent theorem, there is a $\phi\in \A_d$ whose restriction to $Z(\I(X))$ coincides with the characteristic function of $U$. 
In particular, we see that $\phi(z)=0$ for every $z\in X$ while $U$ is disjoint from $Z(\phi)$. 
We conclude that $\phi\in \I(X)$, whence
\[
X\subset Z(\phi)\cap Z(\I(X))\subset Z(\I(X))\setminus U .
\]
This is absurd since $Z(\I(X))$ is the smallest zero set containing $X$.
\end{proof}


We close the paper by exploring an issue related to Corollary \ref{C:zerosetunionTN}, and show that given a closed $\A_d$-totally null subset $K\subset\bS_d$ (which is a zero set by Proposition \ref{P:Z(f)}), we can find a countably infinite set $\Lambda\subset \bB_d$ such that $\Lambda\cup K$ is also a zero set for $\A_d$. For this purpose we will use the notion of an \emph{interpolating sequence}.

Recall that $\Lambda = \{z_n\}_n \subset\bB_d$ is an interpolating sequence for $\M_d$ if the restriction map $\rho_\Lambda:\M_d\to \ell^\infty$ defined as
\[
 \rho_\Lambda(\phi) = \phi|_\Lambda =  \{ \phi(z_n) \}_{n} \qfor \phi\in\M_d 
\]
is surjective. This map is obviously contractive and the norm of the inverse of the induced isomorphism of $\M_d/\ker(\rho_\Lambda)$ onto $\ell^\infty$ is called the \emph{interpolation constant}.
Except in the single variable case of the unit disc, there is no known characterization of interpolating sequences for $\M_d$. On the other hand, such sequences are known to be plentiful. For example, \cite[Proposition~9.1]{DHS15} shows that any sequence in $\bB_d$ converging to the boundary contains an interpolating subsequence.

We first establish a result which is of independent interest, which the reader may want to compare with Corollary \ref{C:interpolation}.

\begin{thm} \label{T:interp}
Let $\Lambda=\{z_n\}_n \subset \bB_d$ be an interpolating sequence for $\M_d$ with interpolation constant $\gamma>0$. Let $K\subset\bS_d$ be a closed $\A_d$-totally null subset containing $\ol{\Lambda}\cap\bS_d$.
Then the restriction map $\rho:\A_d \to \rC(K \cup \Lambda)$ is surjective.
\end{thm}

\begin{proof}
We first show that $\rho$ takes the closed ball of radius $\gamma$ in $\I(K)$ onto a dense subset of the closed unit ball of 
\[
\{f\in \rC(K\cup \Lambda): f|_K=0\}.
\]
Note that since $\ol{\Lambda}\cap \bS_d\subset K$, we can identify elements of that space with sequences in $c_0(\Lambda)$.
Fix $\ep>0$, $N\geq 1$ and $\Ba = \{a_n\}_n \in c_0(\Lambda)$ such that $\|\Ba\|\leq 1$ and $a_n=0$ for $n>N$. Since $\Lambda$ is an interpolating sequence, there is $\psi\in\M_d$ with $\|\psi\|\leq  \gamma$ such that $\psi|_\Lambda = \Ba$. 

Because evaluation at a point in the open unit ball is a weak-$*$ continuous functional on $\M_d$, we may use Corollary~\ref{C:I(K) dense ball} to find $\phi_1\in \I(K)$ such that $\|\phi_1\|\leq \gamma$ and 
\[
|\phi_1(z_n)-a_n|=|\phi_1(z_n)-\psi(z_n)|<\ep/2 \qfor 1\leq n\leq N.
\]
Put $N_1=N$. Note that since $\phi_1\in \I(K)$ and $\ol{\Lambda}\cap \bS_d\subset K$, we must have
\[
 \lim_{n\to\infty} \phi_1(z_n) = 0 .
\]
In particular, there is $N_2>N_1$ such that
\[
|\phi_1(z_n)-a_n|=|\phi_1(z_n)|<\ep/2 \qfor n>N_2.
\]
Proceeding recursively, we obtain a sequence of functions $\{\phi_k\}_k \subset \I(K)$ and a strictly increasing sequence of positive integers $\{N_k\}_k$ such that $\|\phi_k\|\leq \gamma$ and
\[
|\phi_k(z_n)-a_n|<\ep/2 \qfor 1\leq n\leq N_k \qand n>N_{k+1}
\]
for each $k\geq 1$.

Now select a positive integer $M$ so large that $2\gamma< M\ep$.
Define
\[
 \phi = \frac 1 M \sum_{k=1}^M \phi_k .
\]
It is clear that $\phi \in \I(K)$ and $\|\phi\| \leq  \gamma$. Moreover
\[ |\phi(z_n) - a_n| < \frac\ep 2 \qfor 1 \le n \le N \qand n>N_{M+1}\]
because this is valid for each $\phi_k, 1\leq k \leq M$. We claim that this inequality also holds for $N<n\leq N_{M+1}$. Indeed, in that case there is $1\leq j\leq M$ such that $N_j<n\leq N_{j+1}$. Note that if $k\leq j-1$, then $n>N_{k+1}$, so that
\[
|\phi_k(z_n)|=|\phi_k(z_n)-a_n|<\ep/2 ;
\]
while if $k\geq j+1$, then $n\leq N_{k}$, and
\[
|\phi_k(z_n)|=|\phi_k(z_n)-a_n|<\ep/2
\]
as well. Therefore we can estimate for $N<n <N_{M+1}$ that
\[
|\phi(z_n)|\leq \frac{1}{M}\sum_{k\neq j}|\phi_k(z_n)|+\frac{|\phi_j(z_n)|}{M} < \frac\ep 2 + \frac{\gamma}{M} < \ep .
\]
Hence $\| \rho(\phi) - \Ba \| < \ep$. Since $\ep$, $N$ and $\Ba$ were arbitrary, this establishes our claim that the image under $\rho$ of the closed ball of radius $\gamma$ in $\I(K)$ is dense in the closed unit ball of 
\[
\{f\in \rC(K\cup \Lambda): f|_K=0\}.
\]

In turn, arguing as in \cite[Lemma 5.9]{CD} we see that the image under $\rho$ of the open ball of radius $\gamma$ in $\I(K)$ contains the open unit ball of the space above. Moreover, by Theorem~\ref{T:peakinterpolation} (or the more elementary result \cite[Proposition~5.10]{CD}), the restriction of the open unit ball of $\A_d$ to $K$ coincides with the open unit ball of $\rC(K)$.
Fix  now $f \in\rC(K \cup \Lambda)$ with $\|f\| < 1$. Pick $\theta \in\A_d$ such that $\|\theta\| < 1$ and $\theta|_K = f|_K$.
Then $g = f - \theta|_{K\cup \Lambda}\in \rC(K\cup \Lambda)$ vanishes on $K$ and $\|g\| < 2$. Choose $\phi\in\I(K)$ with $\|\phi\| < 2\gamma$ so that $\phi|_{K\cup\Lambda} = g$. Evidently we have $\rho(\theta+\phi)=f$ and we conclude that $\rho$ is surjective.
\end{proof}

Since $\ker \rho=\I(K\cup \Lambda)$, one consequence of this result is that $\A_d/\I(K\cup \Lambda)$ is isomorphic to $\rC(K\cup \Lambda)$. The proof actually shows that the inverse of the isomorphism has norm at most $2\gamma+1$. We also single out a nice consequence for zero sets.

\begin{cor}\label{C:zero interp}
If $\Lambda\subset \bB_d$ is an interpolating sequence for $\M_d$ and $K\subset \bS_d$ is a closed $\A_d$-totally null subset containing $\ol{\Lambda}\cap\bS_d$, then $K \cup \Lambda$ is a zero set for $\A_d$.
\end{cor}

\begin{proof}
By Lemma~\ref{L:characters}, the character space of $\A_d/\I(K\cup \Lambda)$ is naturally identified with $Z(\I(K\cup \Lambda))$.  On the other hand, we have that $\A_d/\I(K\cup \Lambda)$ and  $\rC(K \cup \Lambda)$ are isomorphic as Banach algebras via the restriction map by virtue of Theorem~\ref{T:interp}.  It is easily verified that this implies that 
\[
Z(\I(K\cup \Lambda))=K\cup \Lambda. \qedhere
\]
\end{proof}

\bibliographystyle{amsplain}

\begin{thebibliography}{99}

\bibitem{Arv98} W. Arveson, 
\textit{Subalgebras of $C^*$-algebras III: Multivariable operator theory}, 
Acta Math.\ \textbf{181} (1998), 159--228.

\bibitem{Bishop} E. Bishop, 
\textit{A general Rudin-Carleson theorem},
Proc. Amer. Math. Soc. \textbf{13} (1962), 140--143. 

\bibitem{Carleson57} L. Carleson, 
\textit{Representations of continuous functions}, 
Math. Zeit. \textbf{66} (1957), 447--451.

\bibitem{abscont} R. Clou\^atre and K. Davidson,
\textit{Absolute continuity for commuting row contractions},
to appear in J. Funct. Anal. (2016)

\bibitem{CD} R. Clou\^atre and K. Davidson,
\textit{Duality, convexity and peak interpolation in the Drury-Arveson space},
Adv. Math. 295 (2016), 90--149.


\bibitem{ColeRange} B. Cole and M. Range
\textit{{$A$}-measures on complex manifolds and some applications},
J. Funct. Anal. \textbf{11} (1972), 393--400. 

\bibitem{DHS15} K. Davidson, M. Hartz and O. Shalit,
\textit{Multipliers of Embedded Discs},
Complex Anal.\ Oper.\ Theory \textbf{9} (2015), 287--321.

\bibitem{DP99} K. Davidson and D. Pitts, 
\textit{Invariant subspaces and hyper-reflexivity for free semigroup algebras}, 
Proc. London Math. Soc. (3) \textbf{78} (1999), 401--430. 

\bibitem{DRS2} K. Davidson, C. Ramsey and O. Shalit,
\textit{Operator algebras for analytic varieties}, 
Trans. Amer. Math. Soc. \textbf{367} (2015), 1121--1150.
 
\bibitem{Drury78} S. Drury, 
\textit{A generalization of von Neumann's inequality to the complex ball}, 
Proc.\ Amer.\ Math.\ Soc. \textbf{68} (1978), 300--304.

\bibitem{Gamelin} T. Gamelin,
\textit{Uniform algebras},
Prentice-Hall, 1969.
 
\bibitem{GRS} D. Greene, S. Richter and C. Sundberg
\textit{The structure of inner multipliers on spaces with complete Nevanlinna-Pick kernels}, 
J. Funct. Anal. \textbf{192} (2002), 311--331.

\bibitem{GR} R. Gunning and H. Rossi,
\textit{Analytic functions of several complex variables},
Prentice Hall, 1965.

\bibitem{Hed89} H. Hedenmalm, 
\textit{Closed ideals in the ball algebra},
Bull.\ London Math.\ Soc. \textbf{21} (1989), 469--474. 

\bibitem{Henkin} G. Henkin,
\textit{The Banach spaces of analytic functions in a ball and in a bicylinder are nonisomorphic},
Funkcional. Anal. i Prilo\v zen. \textbf{4} (1968), 82--91.

\bibitem{Hoffman} K. Hoffman,
\textit{Banach spaces of analytic functions}, 
Prentice Hall, 1962.


\bibitem{McCT00} S. McCullough and T. Trent,
\textit{Invariant subspaces and {N}evanlinna-{P}ick kernels}, 
J. Funct. Anal. \textbf{178} (2000), 226--249.

\bibitem{MV93} V. M\"uller and F. Vasilescu,
\textit{Standard models for some commuting multioperators},
Proc.\ Amer.\ Math.\ Soc.\ \textbf{117} (1993), 979--989. 

\bibitem{quiggin93} P. Quiggin,
\textit{For which reproducing kernel Hilbert spaces is Pick's theorem true?}
Integral Equations Operator Theory \textbf{16} (1993), 244--266.

\bibitem{Rudin57} W. Rudin, 
\textit{The closed ideals in an algebra of analytic functions},
Canad.\ J. Math. \textbf{9} (1957), 426--434. 

\bibitem{Rudin} W. Rudin, 
\textit{Function Theory in the Unit Ball of $\bC^n$}, 
Springer-Verlag, 1980.

\bibitem{Valskii} R. Valskii,
\textit{Measures that are orthogonal to analytic functions in {$C^{n}$}},
Dokl. Akad. Nauk SSSR. \textbf{198} (1971), 502--505.

\bibitem{Whitney} H. Whitney,
\emph{Complex analytic varieties},
Addison-Wesley, 1972.

\end{thebibliography}

\end{document}